\documentclass[11pt]{amsart}  
\usepackage{bm}
\usepackage{amscd}
\usepackage{amsmath}
\usepackage{amssymb}
\usepackage{epsf}
\usepackage{verbatim}
\usepackage{graphicx}
\usepackage{mathtools}
\usepackage{enumerate}
\usepackage{color}

\usepackage{amsthm}
\usepackage{thmtools}
\declaretheorem[]{theorem}
\declaretheorem[sibling=theorem]{lemma}

\declaretheorem[style=definition,sibling=theorem]{definition}

\usepackage[usenames,dvipsnames,svgnames,table]{xcolor}
\definecolor{CombinatoricaAqua}{HTML}{00698C}
\definecolor{CombinatoricaBlue}{HTML}{3A3293}
\definecolor{CombinatoricaBrown}{HTML}{66220C}
\definecolor{CombinatoricaRed}{HTML}{DF2A27}
\definecolor{HarvardCrimson}{rgb}{0.6471, 0.1098, 0.1882}
\usepackage{hyperref}
\hypersetup{
	unicode,
	pdfencoding=auto,
	pdfauthor={},
	pdftitle={},
	pdfsubject={},
	pdfkeywords={},
	colorlinks=true,
	citecolor=CombinatoricaBlue,
	linkcolor=CombinatoricaAqua,
	anchorcolor=CombinatoricaBrown,
	urlcolor=HarvardCrimson
  }
\usepackage[capitalize]{cleveref}

\usepackage[nobysame,msc-links]{amsrefs}

\BibSpec{book}{%
  +{}  {\PrintPrimary}                {transition}
  +{,} { \textbf}                     {title} 
  +{.} { }                            {part}
  +{:} { \textit}                     {subtitle}
  +{,} { \PrintEdition}               {edition}
  +{}  { \PrintEditorsB}              {editor}
  +{,} { \PrintTranslatorsC}          {translator}
  +{,} { \PrintContributions}         {contribution}
  +{,} { }                            {series}
  +{,} { \voltext}                    {volume}
  +{,} { }                            {publisher}
  +{,} { }                            {organization}
  +{,} { }                            {address}
  +{,} { \PrintDateB}                 {date}
  +{,} { }                            {status}
  +{}  { \parenthesize}               {language}
  +{}  { \PrintTranslation}           {translation}
  +{;} { \PrintReprint}               {reprint}
  +{.} { }                            {note}
  +{.} {}                             {transition}
  +{}  {\SentenceSpace \PrintReviews} {review}
}

\makeatletter
\renewcommand{\PrintNames@a}[4]{%
  \PrintSeries{\name}
  {#1}
  {}{ and \set@othername}
  {,}{ \set@othername}
  {}{ and \set@othername}
  {#2}{#4}{#3}%
}
\makeatother

\newcommand{\ceil}[1]{\left\lceil#1\right\rceil}

\newcommand\given[1][]{\:#1\vert\:}
\DeclareMathOperator{\polylog}{polylog}
\DeclareMathOperator{\aut}{aut}
\DeclareMathOperator{\exc}{exc}

\DeclarePairedDelimiter{\defaultDelim}{[}{]}

\DeclareMathOperator{\capPr}{\sf Pr}
\renewcommand{\Pr}[2][]{\capPr_{#1}\defaultDelim*{#2}}
\DeclareMathOperator{\capE}{\sf E}
\newcommand{\E}[2][]{\capE_{#1}\defaultDelim*{#2}}
\DeclareMathOperator{\capVar}{\sf Var}
\newcommand{\Var}[2][]{\capVar_{#1}\defaultDelim*{#2}}

\newcommand\cB{\mathcal B}

\newcommand\cE{\mathcal E}

\newcommand{\brac}[1]{\left(#1\right)}
\newcommand{\bfrac}[2]{\left(\frac{#1}{#2}\right)}

\newcommand{\Bernoulli}{\mathsf{Bernoulli}}
\newcommand{\Bin}{\mathsf{Bin}}
\newcommand{\ER}{Erd\H{o}s-R\'enyi}
\newcommand{\End}{\textsc{End}}
\newcommand{\Lrg}{\textsc{Large}}
\newcommand{\Sml}{\textsc{Small}}

\begin{document}

\title{Thresholds in Random Motif Graphs}

\author{Michael Anastos}
\address{Michael Anastos\newline
  Carnegie Mellon University}
\email{
  \href{mailto:manastos@andrew.cmu.edu}
              {manastos@andrew.cmu.edu}}
\urladdr{\href{http://www.math.cmu.edu/~manastos}
              {http://www.math.cmu.edu/~manastos}}

\author{Peleg Michaeli}
\address{Peleg Michaeli\newline
	School of Mathematical Sciences,
	Raymond and Beverly Sackler Faculty of Exact Sciences,
	Tel Aviv University, Tel Aviv 6997801, Israel.}
\email{
  \href{mailto:peleg.michaeli@math.tau.ac.il}
              {peleg.michaeli@math.tau.ac.il}}
\urladdr{
  \href{http://www.math.tau.ac.il/~pelegm}{
        http://www.math.tau.ac.il/~pelegm}}

\author{Samantha Petti}
\address{Samantha Petti\newline
  School of Mathematics,
  Georgia Institute of Technology,
  Atlanta, GA, USA.}
\email{
  \href{mailto:spetti@gatech.edu}
              {spetti@gatech.edu}}
\urladdr{\href{https://math.gatech.edu/people/samantha-petti}
              {https://math.gatech.edu/people/samantha-petti}}

\thanks{S. Petti: This material is based upon work supported by the National Science Foundation Graduate Research Fellowship under Grant No. DGE-1650044}

\begin{abstract}
  We introduce a natural generalization of the \ER{} random graph model in which random instances of a fixed motif are added independently. 
  The \textit{binomial random motif graph} $G(H,n,p)$ is the random (multi)graph obtained by adding an instance of a fixed graph $H$ on each of the copies of $H$ in the complete graph on $n$ vertices, independently with probability $p$.
  We establish that every monotone property has a threshold in this model, and determine the thresholds for connectivity, Hamiltonicity, the existence of a perfect matching, and subgraph appearance.
  Moreover, in the first three cases we give the analogous hitting time results; with high probability, the first graph in the  random motif graph process that has minimum degree one (or two) is connected and contains a perfect matching (or Hamiltonian respectively). 
\end{abstract}

\maketitle

\section{Introduction}\label{sec:intro}
In the late 1950's Gilbert~\cite{Gil59} and Erd\H{o}s and R\'enyi~\cite{ER59} introduced two of the most fundamental models for generating random graphs: the \emph{binomial random graph} $G(n,p)$, generated by independently adding an edge between each pair of vertices in the complete graph on $n$ vertices with probability $p$, and the the \emph{uniform random graph} $G(n,m)$, which is a uniformly chosen graph from all graphs on $n$ vertices with $m$ edges.
Since, the extensive study of these simple constructions has influenced a variety of fields including combinatorics, computer science, and statistical physics (see~\cites{FK,Bol,JLR} for surveys). 

Detailed analysis of the model has led to the development of plethora of new techniques in probability for analyzing random processes, and the model has been used to verify the existence of structures with certain properties~\cite{AS}.
In computer science, the model has been used to analyze the performance of algorithms on an ``average'' case, showing that NP complete problems may be easier random instances.

The rise of data in the form of graphs (e.g.\ internet connections, biological networks, social networks) has further fueled the study of random graphs.
In practice, the comparison of real world networks to the \ER{} model is a popular technique for highlighting the non-random aspects of a network's 
structure~\cites{Yeg04, Alo07,Son05, Mil02}.
Moreover, the model has inspired many other models which are designed to mirror some characteristic of real-world networks (e.g.\ Watts-Strogatz graphs have small diameter~\cite{WT98}, Barab\'asi-Albert preferential attachment graph exhibit a power law degree distribution~\cite{BA99}).

In this paper we consider a natural generalization of the \ER{} model in which  random motifs are added rather than random edges.
A \emph{motif} is a fixed small subgraph, such as a triangle.
The motifs that are overrepresented in a network are correlated to the function of the network~\cites{Yeg04, Alo07,Son05, Mil02}.
Analyzing random graphs formed as the union of many instances of a particular motif $H$ will give insight into the structural properties of networks with many copies of the motif $H$.

We define the \textit{\textbf{binomial} random motif graph} $G(H,n,p)$ as the random (multi)graph obtained by adding an instance of $H$ on each of the $\binom{n}{|V(H)|}\cdot |V(H)|!/\aut(H)$ copies of $H$ in the complete graph on $n$ vertices $K_n$, independently with probability $p$.
Here by $\aut(H)$ we denote the number of automorphisms of $H$.
Note that if $H$ is an edge, then this is exactly $G(n,p)$.
Similarly, the \textit{\textbf{uniform} random motif graph} $\bar{G}(H,n,m)$ is the random (multi)graph obtained by taking the union of $m$ uniformly chosen copies of $H$ in $K_n$ without replacement.

Closely related to $\bar{G}(H,n,m)$ is the \emph{random motif graph process}  $\bar{G}_0(H,n), \bar{G}_1(H,n),\allowbreak ...,\allowbreak \bar{G}_N(H,n)$. 
$\bar{G}_0(H,n)$ is the empty graph on $n$ vertices and for $0\leq i \leq N= \binom{n}{|V(H)|}/\aut(H)$ the graph $\bar{G}_{i+1}(H,n)$ is generated by adding to $\bar{G}_{i}(H,n)$  a copy of $H$, $H_{i+1}$, chosen uniformly at random from all the copies of $H$ except those in $\{H_1,H_2,...,H_i\}$ i.e.\ those that have been added to $\bar{G}_{0}(H,n)$ so far.
Clearly $\bar{G}_m(H,n)$ has the same law as $\bar{G}(H,n,m)$. In addition, by setting $H$ to be an edge we retrieve the random graph process introduce by Erd\H{o}s and R\'enyi~\cite{ER60}.
By considering the random motif graph process in place of the uniform random motif graph model we can phrase results in a finer way (see for example \cref{thm:conn:hit}).

In this work we show that every monotone graph property has a threshold in the binomial random motif graph $G(H,n,p)$.
Then we determine the thresholds for connectivity, existence of a perfect matching, Hamiltoncity and subgraph appearance.
In the first three cases we also show a hitting time result, according to which w.h.p.\footnote{That is, with probability tending to $1$ as $n$ tends to infinity.}\ the first graph in the random motif graph process that has minimum degree one (or two) is connected (or Hamiltonian respectively). 

\subsection{Notation}\label{sec:notation}
Throughout we assume the motif $H$ has no isolated vertices.
For an integer $r\ge 0$, denote by $m_r(H)$ the number of its copies in $K_n$ which intersect the set $[r]$.
For an integer $d\ge 0$  we define the quantities $\delta_d(H)$ and $p_d(H)$ by
\[ \delta_d(H)
  :=\lceil d/\delta(H) \rceil -1
  \hspace{10mm}\text{ and }\hspace{10mm}
  p_d^{\pm}(H):= \frac{\ln n + \delta_d(H)\ln \ln n \pm x(n)}{m_1(H)},\]
where $x(n)$ is any function of $n$ satisfying $1\ll x(n) \ll \ln{\ln{n}}$.
Note that the expected number of added instances of $H$ in $G(H,n,p^\pm_1(H))$ is $m_n(H)\cdot p^\pm_1(H)$, which only depends on $n$ and on $|V(H)|$.

\subsection{Results}\label{sec:results}
A function $p^*=p^*(n)$ is a threshold for a monotone increasing property $\mathcal{P}$ in the random graph $G(H,n,p)$ if
\begin{equation*}
  \lim_{n \to \infty} \Pr{G(H,n,p) \in \mathcal{P}} =
  \begin{cases}
    0 &\text{ if } p/p^* \to 0,\\
    1 &\text{ if } p/p^* \to \infty,
  \end{cases}
\end{equation*}
as $n\to\infty$.
Our first result is a generalization of a theorem by Bollob\'as and Thomason~\cite{BT97}.
\begin{theorem}\label{thm:threshold}
  Every non-trivial monotone graph property has a threshold.
\end{theorem}

Given Theorem \ref{thm:threshold}, a natural goal is to find the thresholds for various monotone properties. The remaining results of this paper are dedicated towards this goal; we determine the threshold for connectivity, the existence of a perfect matchings, Hamiltonicity, and subgraph appearance. 

A first such result, which generalizes a result in~\cite{ER59}, shows, in particular, that the expected number of motifs needed to make the random motif graph connected depends only on the number of (non-isolated) vertices of the motif.

\begin{theorem}\label{thm:conn:th}
  Let $H$ be a fixed graph.  Then
  \begin{equation*}
    \lim_{n\to\infty}\Pr{G(H,n,p)\text{ is connected }} =
    \begin{cases}
      0 & p\le p_1^-(H),\\
      1 & p\ge p_1^+(H).
    \end{cases}
  \end{equation*}
\end{theorem}

In fact, we show a hitting time result, according to which the hitting time of connectivity equals, w.h.p., the hitting time of minimum degree one.
In other words, the random motif graph process becomes connected exactly when the last isolated vertex disappears, with high probability.

Fix an integer $n$ and a graph $H$.
Let $\tau_\mathsf{c}=\min\{i:\bar{G}_i(H,n)\text{ is connected}\}$, and for $d\ge 1$ denote $\tau_d=\min\{i:\delta(\bar{G}_i(H,n))\ge d\}$.

\begin{theorem}\label{thm:conn:hit}
  Let $H$ be a fixed graph. Then w.h.p.\ $\tau_\mathsf{c}=\tau_1$.
\end{theorem}

We remark that if the motif $H$ is connected, every connectivity related question depends solely on the sets of vertices on which copies of $H$ are added, and not on the way they are put there.  Thus, we may model the question as a (binomial or uniform) random $k$-uniform hypergraph, where $k=|V(H)|$. In this case, \cref{thm:conn:th,thm:conn:hit} follow immediately from known results about (loose) connectivity in random hypergraphs (see, e.g.,\cite{Poo15}).

In the following two theorems we show that
 the existence of a perfect matching is also dependent on the number of non-isolated vertices of the motif.
\begin{theorem}\label{thm:match:th}
  Let $H$ be a fixed graph, and assume that $n$ is even. Then,
  \begin{equation*}
    \lim_{n\to\infty}\Pr{G(H,n,p)\text{ has a perfect matching }} =
    \begin{cases}
      0 & p\le p_1^-(H),\\
      1 & p\ge p_1^+(H).
    \end{cases}
  \end{equation*}
\end{theorem}
Let $\tau_\mathsf{M}=\min\{i:\bar{G}_i(H,n)\text{ has a perfect matching}\}$.
The analogue hitting time result is also true.
\begin{theorem}\label{thm:match:hit}
  Let $H$ be a fixed graph, and assume that $n$ is even. Then w.h.p.\ $\tau_\mathsf{M}=\tau_1$.
\end{theorem}
\Cref{thm:ham} establishes that the thresholds for minimum degree $2$ and for Hamiltonicity are the same.
\Cref{thm:ham:hit} shows the hitting time version of that result. 
\begin{theorem}\label{thm:ham}
  Let $H$ be a fixed graph.  Then
  \begin{equation*}
    \lim_{n\to\infty} \Pr{G(H,n,p) \text{ is Hamiltonian }} =
    \begin{cases}
      0 & p\leq p_2^-(H),\\
      1 & p\geq p_2^+(H).
    \end{cases} 
  \end{equation*}
\end{theorem}

Let $\tau_\mathsf{H}:=\min\{i: \bar{G}_i(H,n) \text{ is Hamiltonian}\}$.
\begin{theorem}\label{thm:ham:hit}
  Let $H$ be a fixed graph.  Then w.h.p.\ $\tau_\mathsf{H}=\tau_2$.
\end{theorem}

Next, we describe the threshold for the appearance of a subgraph $S$. If $S$ appears in a random motif graph, then $S$ is a subgraph of some configuration of $b$ copies of $H$ whose union contains $a$ vertices.
For such an $(a,b)$ covering of $S$, we call a subset of the covering containing $b'$ copies of $H$  whose union contains $a'$ vertices an $(a',b')$ subset.
The threshold for the appearance of $S$  depends on $\bar{\gamma}$, the maximum over all covering configurations of the minimum ratio $a'/b'$ for all subsets of the covering configuration.
\Cref{gamma def} formally describes $\bar{\gamma}$.

\begin{theorem} \label{subgraph}
  Let $H$ be a fixed graph, let $S$ be a fixed graph, and set $v=|V(H)|$ and $\bar{\gamma}=\bar{\gamma}(S,H)$. Then 
  \[
  \lim_{n \to \infty} \Pr{ S \subseteq \bar{G}(H,n,m)} = 
  \begin{cases} 0 & m \ll n^{v-\bar{\gamma}}\\
  1 & m \gg n^{v-\bar{\gamma}}.
  \end{cases}
  \]
\end{theorem}

The number of excess edges of a connected graph $S$, or simply its \emph{excess}, is defined to be $\exc(S)=|E(S)|-|V(S)|+1$.  In particular, trees have excess $0$.  We say that $S$ is \emph{unicyclic} if its excess is $1$, or \emph{complex} if its excess is at least $2$.
The following theorem gives a simple description of $\bar{\gamma}$ when the motif $H$ is a path, which allows us to deduce how the copies of $H$ fit together to form a copy of $S$ at the threshold when $S$ first appears.
If $S$ is a tree, a minimal set of edge disjoint copies of $H$ typically forms $S$.  If $S$ is complex, each copy of the path $H$ typically contributes a single edge to $S$.  If it is unicyclic, it may be formed by any edge disjoint configuration of paths $H$.

\begin{theorem}\label{gamma for path}
  Let $H$ be a path of length $v-1$ and let $S$ be a connected graph.
  Let $\beta$ be the minimum number of edge-disjoint copies of $H$ whose union contains $S$ as a subgraph. Let $\eta= \min_{X\subseteq S} \frac{|V(X)|}{|E(X)|}$. Then
  \[
  \bar{\gamma}=
  \begin{cases}
  v-1+1/\beta & \exc(S)=0,\\
  v-1         & \exc(S)=1,\\
  v-2+\eta    & \exc(S)\geq 2.
  \end{cases}
  \]
\end{theorem}

In the case where the motif is a long path, this result establishes a 
connection between the threshold for the appearance of subgraphs in random 
motif graphs and the threshold for the appearance of subgraphs in the trace of 
a random walk on the complete graph $K_n$ (studied in~\cite{KM17}).
Let $S$ be a connected graph and $\beta$ be the minimum number of paths in any edge-disjoint decomposition of $S$ into paths.
If $H$ is longer than the maximum length path 
in such a minimum edge-disjoint path decomposition, then the threshold implied 
by Theorem 9 matches the threshold for the appearance of $S$ in the trace of a 
random walk on the complete graph~\cite{KM17}.

This should not come as a surprise; by noticing that when the motif is a long 
path, the random motif graph model approximates the trace model, in the 
following sense.
One may sequentially ``cut'' the (lazy) simple random walk into chunks with buffers of length $1$.
We delete loops created by the trace of each chunk, and we enforce the condition that the remaining edges  span a path of length $\ell$ (which is fixed but large).
Hence the trace of each such chunk is an independent copy of a path 
of length $\ell$. Thus we may couple the trace model and the random motif model 
such that the trace model will include the random motif model plus some loops plus a small 
number of buffer edges (which gets smaller as $\ell$ gets larger).

Viewing this analogy this way, we may use \cref{subgraph,gamma for path} to 
reprove the main theorems of~\cite{KM17} for the case where the base graph is 
complete.

\section{Existence of thresholds for monotone properties}
\begin{proof}[Proof of \cref{thm:threshold}]
Assume that $\mathcal{P}$ is a monotone increasing property and let $H_1,H_2,\allowbreak ...,\allowbreak H_{m_0(H)}$ be the copies of $H$ that are spanned by $K_n$. 
  Observe that 
  \[
    \Pr{G(H,n,p) \in \mathcal{P}}
      =\sum_{i=0}^{m_0(H)} \sum_{S \in \binom{m_0(H)}{i}} p^{i}(1-p)^{\binom{n}{|V(H)|}-i} \mathbb{I}\bigg(\bigcup_{j \in S} H_j \in \mathcal{P}\bigg)
  \]
  is a polynomial in $p$.
  In addition, since $\mathcal{P}$ is increasing, it is increasing.
  Therefore we may define $p_{1/2}$ by
  \[\Pr{G(H,n,p_{1/2}) \in \mathcal{P}}=\frac{1}{2}.\]
  We will show that $p_{1/2}$ is a threshold for $\mathcal{P}$. For two random graphs $G,G'$ we write $G\subseteq G'$ if $G,G'$ can be coupled such that $G$ is a subgraph of $G'$.
  
  First let $p=\omega(n) p_{1/2}$ where $\omega(n) \to \infty$ as $n \to \infty$ and let $k \in \mathbb{N}$.
  Let $G_i(H,n,p_{1/2})$ be distributed as a $G(H,n,p_{1/2})$ for $i\in [k]$. 
  Then, by considering the probability of no appearance of a fixed copy of $H$, we have that the graph $\cup_{i \in [k]} G_i(H,n,p_{1/2})$ is distributed as $G(H,n, (1-(1-p_{1/2})^k))$.
  Thereafter $1-(1-p_{1/2})^k \leq kp_{1/2}$ implies,  
  \[
    \bigcup_{i \in [k]} G_i(H,n,p_{1/2}) = G(H,n, (1-(1-p_{1/2})^k))
      \subseteq G(H,n,kp_{1/2}).
  \]
  Hence,
  \begin{align*}
    \Pr{G(H,n,\omega(n) p_{1/2}) \in \mathcal{P}} 
    & = 1- \Pr{G(H,n,\omega(n) p_{1/2}) \notin \mathcal{P}}\\
    &\geq \lim_{k \to \infty} 1- \Pr{G(H,n,k p_{1/2}) \notin \mathcal{P}}\\
    & \geq 1- \lim_{k \to \infty}
        \prod_{1=i}^k \Pr{G_i(H,n, p_{1/2}) \notin \mathcal{P}} = 1.
  \end{align*}

  Now assume that $p= p_{1/2}/\omega(n)$ for some $\omega(n) \to \infty$ as $n \to \infty$ and let $k \in \mathbb{N}$.
  Similarly to before, if we let $G_i(H,n,p_{1/2}/\omega(n))$ to be distributed as a $G(H,n,p_{1/2}/\omega(n))$ for $i\in [k]$ then, we have that  
  \begin{align*}
    \bigcup_{i \in [k]} G_i(H,n,p_{1/2}/\omega(n))
    &= G(H,n, (1-(1-p_{1/2}/\omega(n))^k))\\
    &\subseteq G(H,n,kp_{1/2}/\omega(n)) \subseteq G(H,n,p_{1/2}).
  \end{align*}
  Hence,
  \begin{align*}
    \frac{1}{2}=\Pr{G(H,n, p_{1/2}) \in \mathcal{P}} 
    & = 1- \Pr{G(H,n, p_{1/2}) \notin \mathcal{P}}\\
    &\geq \lim_{k \to \infty} 1- \Pr{G(H,n,k p_{1/2}/\omega(n)) \notin \mathcal{P}}\\
    & \geq 1- \lim_{k \to \infty}  \prod_{1=i}^k \Pr{G_i(H,n,p_{1/2}/\omega(n)) \notin \mathcal{P}} \\
    & =1 - \Pr{G_i(H,n, p_{1/2}/\omega(n)) \notin \mathcal{P}}^k.
  \end{align*}
  Rearranging the above gives,
  \begin{equation*}
    \Pr{G_i(H,n, p_{1/2}/\omega(n)) \notin \mathcal{P}}
      \geq \lim_{k \to \infty} \bigg(\frac{1}{2} \bigg)^{1/k} =1.\qedhere
  \end{equation*}
\end{proof}

\section{Connectivity}
\begin{proof}[Proof of \cref{thm:conn:th}]
  If $p\le p_1^-(H)$ then by \cref{thm:deg:th} the minimum degree of $G(H,n,p)$ 
  is w.h.p.\ $0$, hence it is not connected.
  
  Suppose $p\ge p_1^+(H)$.
  In fact, for the argument below, we only assume that 
  $p=(\ln{n}\pm o(\ln{n}))/m_1(H)$ 
  (and the conclusion will follow by monotonicity). 
  Let $k$ denote the number of vertices of $H$.  For $r=1,\ldots,n/2$ denote by 
  $S_r$ the number of connected components of size $r$ in $G(H,n,p)$.  Note 
  that for $r\ge k$, if a set of cardinality $r$ is a connected component, then 
  there exist $\lceil (r-1)/(k-1)\rceil$ copies of $H$ inside the set which
  appear in $G(H,n,p)$, and there are no edges between it and its complement, 
  so none of the $q=q_r(H)$ copies of $H$ that intersect that set appear.
  By \cref{lem:qr},
  \begin{equation*}
    qp \sim rf_k(r/n)\cdot\ln{n} \ge (1+o(1))k\ln{n}.
  \end{equation*}
  Let $\eta=k!/\aut(H)$ and suppose $r\ge k$.  By \cref{lem:fk} and by the 
  union bound there exist constants $c,c',C>0$ depending only on $H$ such that
  \begin{align*}
    \Pr{S_r>0}
    &\le \binom{n}{r} \binom{\eta \binom{r}{k}}{\ceil{\frac{r-1}{k-1}}}
         p^{\ceil{\frac{r-1}{k-1}}}(1-p)^q
    \le \left(\frac{en}{r}\right)^r
         \left(\frac{e\eta\binom{r}{k}p}{\ceil{\frac{r-1}{k-1}}}\right)
             ^{\ceil{\frac{r-1}{k-1}}}
         e^{-qp}\\
    &\le \left[
           C\cdot\frac{n}{r}\cdot r \cdot p^{(r-1)/(r(k-1))} n^{-(1+o(1))k/r}
         \right]^r\\
    &= \left[ C\cdot\polylog{n} \cdot n^{1/r-(1+o(1))k/r} \right]^r
    = o(1).
  \end{align*}
  It follows that
  \begin{align*}
    \Pr{G(H,n,p)\text{ is not connected}}
    &\le \sum_{r=1}^{n/2} \Pr{S_r>0}\\
    &= \Pr{S_1>0} + \sum_{r=k}^{n/2} \Pr{S_r>0}
    = \Pr{S_1>0} + o(1),
  \end{align*}
  but according to \cref{thm:deg:th} (for $p\ge p_1^+(H)$), there are no 
  isolated vertices w.h.p., 
  and the result follows.
\end{proof}

Note that a consequence of this proof is that for 
$p=(\ln{n}\pm o(\ln{n}))/m_1(H)$, with high probability, every connected 
component is of cardinality $1$ or at least $n/2$.  This means that w.h.p.\ 
there exists a unique ``giant'' component of linear size, and the rest of the 
vertices are isolated.  The next lemma, whose proof uses a simple second moment 
argument, estimates the number of these isolated vertices for 
$p_-=(\ln{n}-\ln{\ln{n}})/m_1(H)$.

\begin{lemma}\label{lem:iso}
  The number of isolated vertices in $G(H,n,p_-)$ is w.h.p.\ at most 
  $2\ln{n}$.
\end{lemma}

\begin{proof}
  Let $D_0$ be the number of isolated vertices in $G(H,n,p_-)$.
  First,
  \begin{equation*}
  \E{D_0} = n(1-p_-)^{m_1(H)} \sim ne^{-p_- \cdot m_1(H)}
  = ne^{-\ln{n}+\ln{\ln{n}}} = \ln{n}.
  \end{equation*}
  Moreover,
  \begin{equation*}
    \E{D_0^2} = \E{D_0} + n(n-1)(1-p_-)^{m_2(H)}.
  \end{equation*}
  Denote $L:=2m_1(H)-m_2(H)$.  Thus
  \begin{equation*}
    \E{D_0^2} \le \E{D_0} + \E{D_0}^2(1-p_-)^{-L},
  \end{equation*}
  and since $(1-p_-)^{-L}-1\sim Lp_-$, we have that
  \begin{equation*}
    \Var{D_0} \le \E{D_0} + \E{D_0}^2((1-p_-)^{-L}-1)
    \le \E{D_0} + (L+1)p_-\E{D_0}^2.
  \end{equation*}
  Thus, noting that $Lp_-=o(1)$,
  \begin{align*}
    \Pr{D_0\ge 2\ln{n}}
    &= \Pr{|D_0-\E{D_0}| \ge (1+o(1))\E{D_0}}\\
    &\le (1+o(1))\left(\E{D_0}^{-1}+(L+1)p_-\right) = o(1).\qedhere
  \end{align*}
\end{proof}

\begin{proof}[Proof of \cref{thm:conn:hit}]
  Denote $p_{\pm}=(\ln{n}\pm\ln{\ln{n}})/m_1(H)$ and
  $m_{\pm}=p_{\pm}\cdot m_n(H)$.
  By asymptotic equivalence of the binomial and the uniform models (see, e.g.,~\cite{JLR}*{Section 1.4}) we have that w.h.p.\ $G(H,n,m_-)$ 
  has a unique 
  giant component, and the rest of the connected components are isolated 
  vertices, whose number is at most $2\ln{n}$.  Denote the set of these 
  isolated vertices by $V_0$.  Together with \cref{thm:conn:th} we also 
  conclude that w.h.p.\
  \begin{equation*}
    m_- \le \tau_1 \le \tau_\mathsf{c} \le m_+.
  \end{equation*}
  We may thus couple $\bar{G}(H,n,m_-)$, $\bar{G}(H,n,\tau_1)$, $\bar{G}(H,n,\tau_\mathsf{c})$ and 
  $\bar{G}(H,n,m^+)$ such that
  \begin{equation*}
    \bar{G}(H,n,m_-)\subseteq \bar{G}(H,n,\tau_1)\subseteq \bar{G}(H,n,\tau_\mathsf{c})\subseteq 
    \bar{G}(H,n,m_+),
  \end{equation*}
  by starting with $\bar{G}(H,n,m_-)$ and adding $M=m_+-m_-$ random copies of $H$ to 
  create $\bar{G}(H,n,m_+)$.  Note that if none of these $M$ edges is fully contained 
  in $V_0$ (and the coupling succeeds) then $\tau_1=\tau_\mathsf{c}$.  Thus, there exist 
  positive constants $C_1,C_2$ such that,
  \begin{equation*}
    \Pr{\tau_1 < \tau_\mathsf{c}} \le o(1) + M\cdot 
    \frac{C_1\binom{|V_0|}{k}}{m_n(H)-m_+}
    \le o(1) + C_2\cdot \frac{m_n(H)\ln{\ln{n}}}{m_1(H)}
      \cdot \frac{\ln^2{n}}{m_n(H)} = o(1).\qedhere
  \end{equation*}
\end{proof}

\section{Hamiltoncity and Perfect Matchings}
The proof of Theorems \ref{thm:ham:hit} and \ref{thm:match:hit} can be given in parallel, using the same techniques and tools.
For clarity though, in this section we focus mainly on proving  \cref{thm:ham:hit}  and we give a sketch of the proof of  \cref{thm:match:hit} in the appendix.

For proving our Hamiltonicity result we use the standard technique of Posa's rotations.
We define \Sml{} to be the vertices of significantly smaller degree than the expected one and we set $\Lrg$ to be the rest of the vertices. 
We first show that small to medium  subsets of $\Lrg$ expand and that the vertices in $\Sml$ are well spread. This is done in the context of Lemmas \ref{lem:expansion} and \ref{smallapart}, \ref{smallapart2} respectively. We use these properties of $\Sml$ and $\Lrg$ in order to prove all the the ingredients needed to apply the Posa's rotations, which we gather in \cref{summary}. 

Let $p_0:=(\ln n-2\ln \ln n)/m_1(H)$ and recall that  $p_2^{\pm}=(\ln n+r_2\ln \ln n \pm \omega(1))/m_1(H)$, $r_2=\lfloor 2/\delta(H)-1\rfloor$. 
W.h.p.\ (see~\cite{FK}) we can couple 
$G(H,n,p_0), G(H,n,p_2^-), \bar{G}(H,n,\tau_2)$ and $G(H,n,p_2^+)$ such that
\begin{itemize}
  \item[](i) $G(H,n,p_0)\subset G(H,n,p_2^-) \subset \bar{G}(H,n,\tau_2)\subset G(H,n,p_2^+)$ and 
  \item[](ii) there are $(1+o(1))(p_2^- -p_0) \frac{r!}{\aut(H)} \binom{n}{r} >  
  n \ln \ln n/2r $ copies of $H$ in $G(H,n,p_2^-)$, hence in $\bar{G}(H,n,\tau_2)$, 
  that are not present in $G(H,n,p_0)$.
\end{itemize}

Observe that the above coupling and  \cref{thm:ham:hit} imply  \cref{thm:ham}.
In addition a similar coupling and  \cref{thm:match:hit} imply 
\cref{thm:match:th}.

We now define the sets $\Sml$, $\Lrg$ based on the degrees of the vertices in $G(H,n,p_0)$.
Let $\Lrg=\{v\in V: v \text{ intersects at least} \ln \ln n \text{ copies of $H$ in  } G(H,n,p_0)\}$ and $\Sml=V\setminus \Lrg$.
\begin{lemma}\label{lem:expansion}
  W.h.p. every $S\subset \Lrg$ of size at most $n/30r$ satisfies
  $|N(S)|\geq 10|S|$.
\end{lemma}

\begin{lemma}\label{smallapart}
  W.h.p.\@ for every pair $u,v \in \Sml$ there do not exist $\ell\leq 6$ copies of $H$ in $G(H,n,p_2^+)$ that span a connected subgraph containing both $u,v$. Hence w.h.p.\@ every pair $u,v \in \Sml$
  is at distance at least 7 in $G(H,n,p_2^+)$.
\end{lemma}

\begin{lemma}\label{smallapart2}
  W.h.p. for every $v \in V$
  there exists at most one copy of $H$ in $G(H,n,p_2^+)$, hence in $\bar{G}(H,n,\tau_2)$, that intersect  both $\{v\}$ and  $\Sml\setminus\{v\}$.
\end{lemma}

Now we generate $ \bar{G}(H,n,\tau_2)$ as follows. We first generate 
$G_0'=G(H,n,p_0)$. Then we randomly permute the copies of $H$ not appearing in $G_0'$, let them be $H_1,H_2,....$. We also let $S_0=\emptyset$.
We define the sequences $G_0',G_1',...$ and $S_0,S_1,...$ in the following way.
At step $i\in \mathbb{N}$ we query $H_i$ whether it is incident 
to a vertex in $\Sml$. If it is then we set
$S_i=S_{i-1}$ and $G_i'=G_{i-1}'\cup H_i$. Otherwise we set 
$S_i=S_{i-1}\cup \{H_i\}$ and $G_i'=G_{i-1}'$.
Let $t^*=\min\{i:\delta(G_i')=2\}$ and $S_{t^*}=\{H_{i_1},H_{i_2},...,H_{i_w}\}$.

Given the sequence $G_0',G_1',...,G_{t^*}'$ and the set  $S_{t^*}=\{H_{i_1},H_{i_2},...,H_{i_w}\}$
we  define the graph sequence  $F_0,...,F_w$ by $F_0=G_{t^*}'$
and $F_j=F_{j-1}\cup H_{i_j}$ for $1\leq j\leq w$.
Observe that $S_{t^*}$ consists of all copies of $H$ in $\{H_1,...,H_{t^*}\}$ that have not been added to $G_0'$, equivalently the copies of $H$ that are not incident to $\Sml$. Thus 
$F_w=G'_{t^*} \cup \big( \bigcup_{j=1}^wH_{i_j} \big) = G'_0 \cup \big( \bigcup_{i=1}^{t^*}H_{i} \big)= \bar{G}(H,n,\tau_2).$

\begin{lemma}\label{summary}
  W.h.p.\@ the following hold:
  \begin{enumerate}[i)]
    \item $w\geq n\ln \ln n /2r-n$,
    \item every $S\subset V$ of size at most $n/30r$ satisfies
    $|N(S)|\geq 2|S|$ in  $F_0$,
    \item $F_0$ is connected,
    \item for every $1\leq j\leq w$, $\epsilon>0$, and every set $Q_j$ consisting of 
    $\epsilon n^2$ edges not present in $F_j$ there exist a constant $C_\epsilon>0$ such that the probability that $Q_j$ intersects $E(H_{i_{j+1}})$ is at least $C_\epsilon$.
  \end{enumerate}
\end{lemma}

We are now ready to apply Posa's rotations . 
For that assume that $F_j$ is not Hamiltonian and
consider a longest path in $F_j$, $P_j$, $j\geq 0$.
Let $x,y$ be the 
end-vertices of $P_j$. Given $yv$ where $v$ is an interior vertex of $P_j$ we can obtain a new longest path $P_j' = x..vy..w$ where $w$ is the neighbor of $v$ on 
$P_j$ between $v$ and $y$. In such a case we say that $P_j'$
is obtained from $P_j$ by a rotation with the end-vertex $x$ being the fixed end-vertex.

Let $\End_j(x;P_j)$ be the set of end-vertices of longest paths of $F_j$ 
that can be obtained from $P_j$ by a sequence of rotations that keep $x$ as
the fixed end-vertex.
Thereafter for $z\in \End_j(x;P_j)$ let $P_j(x,z)$ be a path that has end-vertices $x,z$ and can be obtain form $P_j$ by a sequence of rotations that keep $x$ as
the fixed end-vertex. 
Observe that for $z\in \End_j(x;P_j)$ and 
$z'\in \End_j(z;P_j(x,z))$ there exists a $z$-$z'$ path $P_{z,z'}$ of length $|P_j|$ that can be obtained from $P_j$ via a sequence of Posa rotations. Thus we can conclude that $\{z,z'\}$ does not belong to $F_j$. Indeed assume that $\{z,z'\} \in E(G_i)$. Then we can close $P_{z,z'}$  into a cycle $C_{z,z'}$ that is not Hamiltonian. Since $F_j$ is connected there is an edge $e$ spanned by $V(C_{z,z'})\times V\setminus V(C_{z,z'})$. $E(C_{z,z'}) \cup \{e\}$ spans a path of length  $|P_j|+2$ contradicting  the maximality of $P_j$.
Similarly if $\{z,z'\}\in E(H_{i_{j+1}})$ then $F_{j+1}$ is either Hamiltonian or it contains a path that is longer than $P_j$.
At the same time it follows (see~\cite{FK}*{Corollary 6.7}) that 
\[|N(\End(x,P_j))| < 2|\End(x,P_j)|.\]
Moreover for every 
$z\in \End_j(x;P_j)$
\[|N(\End(z,P_j(x,z)))| < 2|\End(z,P_j(x,z))|.\]
As a consequence of \cref{lem:expansion},
we have that $|\End(x,P_j)| \geq n/30r$ and  
$ |\End(z,P_j(x,z))|$ $\geq n/30r$ for every 
$z\in \End_j(x;P_j)$. 
Let $E_j=\{\{z,z'\}: z\in \End_j(x;P_j) \text{ and } z'\in \End_j(z;P_j(x,z))\}$. Then $|E_j|\geq (n/30r)^2/2$.
\vspace{3mm}
\\Now let $Y_{j}$ be the indicator of the event  $\{E_j\cap E(H_{i_{j+1}}) \neq \emptyset\}$ 
and  set $Z=\sum_{j=1}^{w} Y_i$.
From \cref{summary} iv) we have $\Pr{Y_j=1}\geq C_\epsilon$ (here $\epsilon =1/2(30r)^2$).
In the event that $G_w$ is not Hamiltonian, $Z\leq n$ while 
$Y_j$ is a $\Bernoulli(C_\epsilon)$ random variable for $1\leq j\leq w$ . Since $w\geq n\ln \ln n/2r-n$ we have
$\Pr{\Bin(w,C_\epsilon)\leq n }=o(1)$.
Hence w.h.p.\@ $F_w=\bar{G}(H,n,\tau_2)$ is Hamiltonian and the hitting time for Hamiltonicity equals the hitting time for minimum degree 2. 

\section{Subgraph appearance}
In $G(n,p)$ there is only one way for a specified subgraph to appear on a fixed set of vertices: all the edges in the subgraph must be present.
In the case of random motif graphs, there are multiple ways to place motifs so that a specified subgraph appears on a fixed set of vertices.
For example, in a random two-path graph, a triangle may appear on $\{1,2,3\}$ if (i) the paths $(1,2,3)$ and $(3,1, z)$ are present or (ii) the paths $(1,2,x)$, $(2,3,y)$ and $(3,1,z)$ are present.
In order to pin down the threshold for subgraph appearance, it is necessary to understand the various motif configurations that cause the subgraph to appear and their relative probabilities.
The following definition provides the notation to describe such configurations.

\begin{definition} \label{gamma def} Let $V$ be a set of vertices. Let $S$ be a fixed graph on a subset of the vertices of $V$. Let $H_1, H_2, \dots H_b$ be copies $H$ also defined on subsets of vertices of $V$.
  \begin{enumerate}[(a)]
    \item We say $\{H_1, H_2, \dots H_b\}$ is an $(a,b)$ covering of $S$ if (i) $S \subseteq \bigcup_{j=1}^b H_j$, (ii) $|V(\bigcup_{j=1}^b H_j)|=a$, and (iii) for each $\ell \in [b]$, $S \not \subseteq \bigcup_{j=1}^b H_j \setminus H_\ell$. 
    
    \item Let $k(a,b)$ be the number of unique configurations of $(a,b)$ coverings, i.e. the number of ways to place $b$ copies of $H$ on $a$ vertices such that conditions (i)-(iii) of (a) hold. Enumerate the possible configurations of $(a,b)$ coverings with values in $[k(a,b)]$. For $i \in [k(a,b)]$, an $(a,b,i)$ covering of $S$ is an $(a,b)$ covering with configuration $i$. 
    
    \item We say the set $\{F_1, F_2, \dots F_{b'}\}$ (with precisely $b'$ elements) is an $(a',b')$ subset of an $(a,b,i)$ covering $\{H_1, H_2, \dots H_b\}$ if (i) $\{F_1, F_2, \dots F_{b'}\} \subseteq \{H_1, H_2, \dots H_{b}\}$, and (ii) $|V(\bigcup_{\ell=1}^{b'} F_\ell)|=a'$. 

    \item Let $\mathcal{I}(S, H)= \{(a,b,i) \given \text{there exists an $(a,b)$ covering of $S$ by $H$ and $i \in [k(a,b)]$}\}.$ 
    
    \item For $(a,b,i ) \in \mathcal{I}(S, H)$, let \[\mathcal{D}(a,b,i)= \{ (a', b') \given \text{there exists an $(a',b')$ subset of the  $(a,b,i)$ covering}\}.\]
    
    \item For $(a,b,i) \in  \mathcal{I}(S, H)$,
    let $\gamma(a,b,i)= \min_{(a',b') \in \mathcal{D}(a,b,i)} \frac{a'}{b'}$
    and denote \[\bar{\gamma}= \max_{(a,b,i) \in \mathcal{I}(S,H)} \gamma(a,b,i).\]
  \end{enumerate}
\end{definition}

\begin{proof}[Proof of \Cref{subgraph}]
  Let $G\sim \bar{G}(H,n,m)$.	We say that an instance of the subgraph $S$ in $G$ is an $(a,b,i)$ instance if the placed graphs $H_1, \dots H_b$ that contribute at least one edge to $S$ form an $(a,b,i)$ covering of $S$.
  Let $X_i^{ab}$ denote the number of $(a,b,i)$ instances of $S$ in $G$.  Let $Z= \sum_{(a,b,i) \in \mathcal{I}(S,H)}  X_i^{ab}$ be the total number of instances of the subgraph $S$ in $G$. 
  
  First we use the first moment method to show that if $m \ll n^{v-\bar{\gamma}}$, then the probability that $S$ occurs as a subgraph is $o(1)$. It suffices to show that for all $(a,b,i) \in \mathcal{I}(S,H)$, $\E{X_i^{ab}}=o(1)$ since 
  \[\Pr{Z >0} \leq \E{Z}= \sum_{(a,b,i) \in \mathcal{I}(S,H)}
  X_{i}^{ab},\] and $|\mathcal{I}(S,H)|$ is a constant independent of $n$.
  
  We now compute $\E{X_i^{ab}}$ for a fixed triple $(a,b,i) \in \mathcal{I}(S,H)$. Let $\{F_1, \dots F_{b'}\}$  be an $(a',b')$ subset of the configuration $(a,b,i)$  with  $a'/b'= \gamma(a,b,i)$.
  Let $Y$ be the number of instances of  $F=\bigcup_{i=1}^{b'} F_{b'}$ in $G$ formed by the configuration $\{F_1, \dots F_{b'}\}$. Since an $(a,b,i)$ instance of $S$ contains an instance of the configuration $\{F_1, \dots F_{b'}\}$, $X_i^{ab} \leq Y$. The number of ways to select $a'$ vertices is at most $n^{a'}$. The probability that a labeled copy of $H$ is placed on a specified set of vertices is $m/n^v$. 
  We compute
  \[\E{X_i^{ab}} \leq \E{Y} \leq c n^{a'} \bfrac{m}{n^v}^{b'}=  c\brac{n^{\gamma(a,b,i)-v}m}^{b'} \leq c\brac{n^{\bar{\gamma}-v}m}^{b'},\]
  where $c$ is a constant depending only on the number of automorphisms of $S$ and the number of automorphisms of the configuration $\{F_1, \dots F_{b'}\}$.
  It follows that for $m \ll n^{\bar{\gamma}-v}$, $\E{X_i^{ab}}=o(1)$, as desired.
  
  Next we use the second moment method to show that if $m \gg n^{v - \bar{\gamma}}$ then $S$ appears as a subgraph almost surely. It suffices to show that there exists some $(a,b,i) \in \mathcal{I}(S,H)$ such that $X_i^{ab}$ is almost surely positive. 
  Let $(a,b,i)$ be such that $\bar{ \gamma}= \gamma(a,b,i)$. We apply Corollary~4.3.5 of~\cite{AS} to show that  $X_i^{ab}$ is almost surely positive. Let $X_i^{ab}= \sum_j A_j$ where $A_j$ is an indicator random variable for the event that there is an $(a,b,i)$ instance of $S$ formed by a configuration of $H_1, H_2, \dots H_b$ each present on a specified set of vertices. Fix $A_\ell$, and let
  \[\Delta^\ast= \sum_{j \sim \ell}  \Pr{A_j \given A_\ell},\] where $j \sim \ell$ indicates that $A_j$ and $A_\ell$ are not independent. By 4.3.5 of~\cite{AS}, if $\E{X_{i}^{ab}} \to \infty $ and $\Delta^\ast=o(\E{X_i^{ab}})$, then $X_i^{ab} >0$ almost surely.
  
  First we show that $\E{X_{i}^{ab}} \to \infty $. We compute as above
  \[\E{X_i^{ab}}\geq c' n^a \bfrac{m}{n^v}^b= c'\brac{n^{a/b-v}m}^b \geq c' \brac{n^{\bar{\gamma}-v}m}^b\]
  where $c'$ is a constant depending only on the number of automorphisms of $S$ and the number of automorphisms of the configuration $\{H_1, \dots H_{b}\}$. It follows that if $m \gg n^{v - \bar{\gamma}}$ then $ \E{X_i^{ab}}\to \infty$.
  
  Finally, we show $\Delta^\ast=o(\E{X_i^{ab}})$.
  Observe that under the assumption $m \gg n^{v - \bar{\gamma}}$, \begin{align*}
  \Delta^*&= \sum_{(a', b') \in \mathcal{D}(a,b,i)} c n^{a-a'} \bfrac{m}{n^v}^{b-b'}=\sum_{(a', b') \in \mathcal{D}(a,b,i)} c \E{X_i^{ab}} \brac{ n^{-a'/b'+v} m^{-1}}^{b'} \\
  &\leq c' \E{X_i^{ab}}\brac{ n^{-\gamma(a,b,i)+v} m^{-1}}^{b}=c' 
  \E{X_i^{ab}}\brac{ n^{v-\bar{\gamma}} m^{-1}}^{b}= o \brac{\E{X_i^{ab}}}.
  \qedhere
  \end{align*}
\end{proof}

\section{Conclusion}
\subsection{The value of the random motif model}

The study of random motif graphs has the potential to strengthen the impact of 
the \ER{} construction. In the context of analyzing real-world 
networks with an overrepresented motif, random motif graphs may be a more 
insightful null hypothesis model to compare against to identify non-random 
structure. For instance by studying subgraphs counts of random $H$ motif graphs 
one can determine if some  larger motif pattern is a byproduct of having many 
copies of $H$ or is itself some novel aspect of the network structure. 
Moreover, it is possible that a random motif graph may be used to establish the 
existence of a graph with some extremal property of interest. Finally, random 
motif graphs can be used as an alternate definition of average case for 
analyzing algorithms under the assumption that the input has some motif 
structure. 

\subsection{Future directions: understanding threshold behavior more broadly}
We have established that random motif graphs behave similarly to traditional 
\ER{} random graphs with respect to thresholds and hitting times 
for monotone properties. Does similar behavior appear when we consider random 
graphs formed by randomly adding primitive subgraphs $H$ whose size scales with 
$n$, the number of vertices of the random graph? Instead of taking $H$ to be a 
fixed motif, $H$ could be a path, cycle, matching or clique whose size depends 
on $n$, for example.  Some of these cases were in fact studied in several 
contexts. For example, the union of $d\ge 3$ random perfect matchings is 
contiguous to the random $d$-regular graph, and is sometimes easier to 
analyze~\cite{Wor99}.
Moreover, we can consider the class of models where $H$ itself is chosen from some probability distribution.  In several cases, this has been 
studied as well.  For instance, \cite{FKT99} and~\cite{FGRV14} consider 
the case when $H$ is the uniform spanning tree, and~\cite{Pet18} considers the 
case when $H$ is an \ER{} random graph with constant density and 
size dependent on $n$. Further study of 
these models is a first step toward delineating a larger family of random 
graphs that exhibit \ER{} like threshold and hitting time behaviors.

\bibliographystyle{plain}
\bibliography{randommotifs}

\begin{bibdiv}
\begin{biblist}

\bib{AS}{book}{
      author={Alon, Noga},
      author={Spencer, Joel~H.},
       title={The probabilistic method},
     edition={Fourth},
      series={Wiley Series in Discrete Mathematics and Optimization},
   publisher={John Wiley \& Sons, Inc., Hoboken, NJ},
        date={2016},
        ISBN={978-1-119-06195-3},
      review={\MR{3524748}},
}

\bib{Alo07}{article}{
      author={Alon, Uri},
       title={Network motifs: theory and experimental approaches},
        date={2007},
     journal={Nature Reviews Genetics},
      volume={8},
      number={6},
       pages={450\ndash 461},
}

\bib{BA99}{article}{
      author={Barab\'{a}si, Albert-L\'{a}szl\'{o}},
      author={Albert, R\'{e}ka},
       title={Emergence of scaling in random networks},
        date={1999},
        ISSN={0036-8075},
     journal={Science},
      volume={286},
      number={5439},
       pages={509\ndash 512},
         url={https://doi.org/10.1126/science.286.5439.509},
      review={\MR{2091634}},
}

\bib{Bol}{book}{
      author={Bollob\'{a}s, B\'{e}la},
       title={Random graphs},
   publisher={Academic Press, Inc. [Harcourt Brace Jovanovich, Publishers],
  London},
        date={1985},
        ISBN={0-12-111755-3; 0-12-111756-1},
      review={\MR{809996}},
}

\bib{BT97}{incollection}{
      author={Bollob\'{a}s, B\'{e}la},
      author={Thomason, Andrew},
       title={Hereditary and monotone properties of graphs},
        date={1997},
   booktitle={The mathematics of {P}aul {E}rd\H{o}s, {II}},
      series={Algorithms Combin.},
      volume={14},
   publisher={Springer, Berlin},
       pages={70\ndash 78},
         url={https://doi.org/10.1007/978-3-642-60406-5_7},
      review={\MR{1425205}},
}

\bib{ER59}{article}{
      author={Erd\H{o}s, Paul},
      author={R\'{e}nyi, Alfr\'ed},
       title={On random graphs. {I}},
        date={1959},
        ISSN={0033-3883},
     journal={Publicationes Mathematicae Debrecen},
      volume={6},
       pages={290\ndash 297},
      review={\MR{0120167}},
}

\bib{ER60}{article}{
      author={Erd\H{o}s, Paul},
      author={R\'{e}nyi, Alfr\'ed},
       title={On the evolution of random graphs},
        date={1960},
     journal={Magyar Tud. Akad. Mat. Kutat\'{o} Int. K\"{o}zl.},
      volume={5},
       pages={17\ndash 61},
      review={\MR{0125031}},
}

\bib{FGRV14}{article}{
      author={Frieze, Alan},
      author={Goyal, Navin},
      author={Rademacher, Luis},
      author={Vempala, Santosh},
       title={Expanders via random spanning trees},
        date={2014},
        ISSN={0097-5397},
     journal={SIAM Journal on Computing},
      volume={43},
      number={2},
       pages={497\ndash 513},
         url={https://doi.org/10.1137/120890971},
      review={\MR{3181708}},
}

\bib{FK}{book}{
      author={Frieze, Alan},
      author={Karo\'{n}ski, Micha{\l}},
       title={Introduction to random graphs},
   publisher={Cambridge University Press, Cambridge},
        date={2016},
        ISBN={978-1-107-11850-8},
         url={https://doi.org/10.1017/CBO9781316339831},
      review={\MR{3675279}},
}

\bib{FKT99}{article}{
      author={Frieze, Alan},
      author={Karo\'{n}ski, Micha{\l}},
      author={Thoma, Lubo\v{s}},
       title={On perfect matchings and {H}amilton cycles in sums of random
  trees},
        date={1999},
        ISSN={0895-4801},
     journal={SIAM Journal on Discrete Mathematics},
      volume={12},
      number={2},
       pages={208\ndash 216},
         url={https://doi.org/10.1137/S0895480196313790},
      review={\MR{1686834}},
}

\bib{Gil59}{article}{
      author={Gilbert, Edgar~N.},
       title={Random graphs},
        date={1959},
        ISSN={0003-4851},
     journal={Annals of Mathematical Statistics},
      volume={30},
       pages={1141\ndash 1144},
         url={https://doi.org/10.1214/aoms/1177706098},
      review={\MR{108839}},
}

\bib{JLR}{book}{
      author={Janson, Svante},
      author={{\L}uczak, Tomasz},
      author={Rucinski, Andrzej},
       title={Random graphs},
      series={Wiley-Interscience Series in Discrete Mathematics and
  Optimization},
   publisher={Wiley-Interscience, New York},
        date={2000},
        ISBN={0-471-17541-2},
         url={https://doi.org/10.1002/9781118032718},
      review={\MR{1782847}},
}

\bib{KM17}{article}{
      author={Krivelevich, Michael},
      author={Michaeli, Peleg},
       title={Small subgraphs in the trace of a random walk},
        date={2017},
        ISSN={1077-8926},
     journal={Electronic Journal of Combinatorics},
      volume={24},
      number={1},
       pages={Paper 1.28, 22},
      review={\MR{3625905}},
}

\bib{Mil02}{article}{
      author={Milo, Ron},
      author={Shen-Orr, Shai},
      author={Itzkovitz, Shalev},
      author={Kashtan, Nadav},
      author={Chklovskii, Dmitri~B.},
      author={Alon, Uri},
       title={Network motifs: simple building blocks of complex networks},
        date={2002},
     journal={Science},
      volume={298},
      number={5594},
       pages={824\ndash 827},
}

\bib{Pet18}{article}{
      author={{Petti}, Samantha},
      author={{Vempala}, Santosh~S.},
       title={{Approximating Sparse Graphs: The Random Overlapping Communities
  Model}},
        date={2018Feb},
     journal={arXiv e-prints},
       pages={arXiv:1802.03652},
      eprint={1802.03652},
}

\bib{Poo15}{article}{
      author={Poole, Daniel},
       title={On the strength of connectedness of a random hypergraph},
        date={2015},
        ISSN={1077-8926},
     journal={Electronic Journal of Combinatorics},
      volume={22},
      number={1},
       pages={Paper 1.69, 16},
      review={\MR{3336583}},
}

\bib{Son05}{article}{
      author={Song, Sen},
      author={Sj{\"o}str{\"o}m, Per~Jesper},
      author={Reigl, Markus},
      author={Nelson, Sacha~B.},
      author={Chklovskii, Dmitri~B.},
       title={Highly nonrandom features of synaptic connectivity in local
  cortical circuits},
        date={2005},
     journal={PLoS biology},
      volume={3},
      number={3},
       pages={e68},
}

\bib{WT98}{article}{
      author={Watts, Duncan~J.},
      author={Strogatz, Steven~H.},
       title={Collective dynamics of ``small-world" networks},
        date={1998},
     journal={Nature},
      volume={393},
      number={6684},
       pages={440\ndash 442},
         url={https://doi.org/10.1038/30918},
}

\bib{Wor99}{incollection}{
      author={Wormald, Nicholas~C.},
       title={Models of random regular graphs},
        date={1999},
   booktitle={Surveys in combinatorics, 1999 ({C}anterbury)},
      series={London Math. Soc. Lecture Note Ser.},
      volume={267},
   publisher={Cambridge Univ. Press, Cambridge},
       pages={239\ndash 298},
      review={\MR{1725006}},
}

\bib{Yeg04}{article}{
      author={Yeger-Lotem, Esti},
      author={Sattath, Shmuel},
      author={Kashtan, Nadav},
      author={Itzkovitz, Shalev},
      author={Milo, Ron},
      author={Pinter, Ron~Y.},
      author={Alon, Uri},
      author={Margalit, Hanah},
       title={Network motifs in integrated cellular networks of
  transcription-regulation and protein-protein interaction},
        date={2004},
     journal={Proceedings of the National Academy of Sciences of the United
  States of America},
      volume={101},
      number={16},
       pages={5934\ndash 5939},
}

\end{biblist}
\end{bibdiv}

\newpage
\appendix

\section{Estimates for useful functions}

\begin{lemma}\label{lem:mr}
	For $r=r(n)$, if $k=|V(H)|$ and $\alpha=r/n$ then
	$m_r(H)\sim rm_1(H)\cdot \frac{1-(1-\alpha)^k}{k\alpha}$.
\end{lemma}

\begin{proof}
	Observe that for $r\ge 0$,
	\begin{align*}
	m_r(H) = \left(\binom{n}{k} - \binom{n-r}{k}\right)\cdot
	\frac{k!}{\aut(H)},
	\end{align*}
	thus
	\begin{equation*}
	\frac{m_r(H)}{rm_1(H)} =
	\frac{\binom{n}{k} - \binom{n-r}{k}}
	{r\left(\binom{n}{k} - \binom{n-1}{k}\right)}
	\sim \frac{n^k-(n-r)^k}{r(n^k-(n-1)^k)}
	= \frac{1-(1-\alpha)^k}{r(1-(1-n^{-1})^k)}
	\sim \frac{1-(1-\alpha)^k}{k\alpha}.\qedhere
	\end{equation*}
\end{proof}

For $r\ge 1$, denote by $q_r(H)$ the number of copies of $H$ that intersect 
$[r]$ but that are not contained in $[r]$.

\begin{lemma}\label{lem:qr}
	For $r=r(n)$, if $k=|V(H)|$ and $\alpha=r/n$ then
	$q_r(H)\sim rm_1(H)\cdot \frac{1-(1-\alpha)^k-\alpha^k}{k\alpha}$.
\end{lemma}

\begin{proof}
	Observe that for $r\ge 0$,
	\begin{align*}
	q_r(H) = \left(\binom{n}{k}-\binom{n-r}{k}-\binom{r}{k}\right)
	\cdot \frac{k!}{\aut(H)},
	\end{align*}
	thus
	\begin{align*}
	\frac{q_r(H)}{rm_1(H)} =
	\frac{\binom{n}{k} - \binom{n-r}{k} - \binom{r}{k}}
	{r\left(\binom{n}{k} - \binom{n-1}{k}\right)}
	&\sim \frac{n^k-(n-r)^k-r^k}{r(n^k-(n-1)^k)}\\
	&= \frac{1-(1-\alpha)^k-\alpha^k}{r(1-(1-n^{-1})^k)}
	\sim \frac{1-(1-\alpha)^k-\alpha^k}{k\alpha}.\qedhere
	\end{align*}
	
\end{proof}

For convenience we define for $\alpha\in[0,1]$ and $k\ge 1$, 
\begin{equation*}
f_k(\alpha) = \frac{1-(1-\alpha)^k-\alpha^k}{k\alpha}.
\end{equation*}

\begin{lemma}\label{lem:fk}
	For $2\le k\le r$ we have that $rf_k(r/n) \ge (1+o(1))k$.
\end{lemma}

\begin{proof}
	Write $g_k(\alpha)=f_k(\alpha)\cdot k\alpha=1-(1-\alpha)^k-\alpha^k$.  
	Observe that it is strictly increasing in $(0,1/2)$.  Note also that
	\begin{equation*}
	n\cdot g_k\left(\frac{k}{n}\right)
	= n-ne^{-k^2/n}- o(1)
	\sim k^2.
	\end{equation*}
	It follows that
	\begin{equation*}
	\frac{kr}{n} \cdot f_k\left(\frac{r}{n}\right)
	= g_k\left(\frac{r}{n}\right)
	\ge g_k\left(\frac{k}{n}\right)
	\sim \frac{k^2}{n},
	\end{equation*}
	so $rf_k(r/n) \ge (1+o(1))k$.
\end{proof}

\section{Minimum degree}
\begin{theorem}\label{thm:deg:th}
	With high probability
	\begin{equation*}
	\delta(G(H,n,p_d^-)) < d\qquad\text{and}\qquad
	\delta(G(H,n,p_d^+)) \ge d.
	\end{equation*}
\end{theorem}

\begin{proof}
	Let $\delta=\delta(H)$.  It suffices to show that with high probability for 
	$\ell \in \mathbb{Z}_{\geq 0}$ 
	\begin{align}\label{min:low}
	\Pr{\delta(G(H,n,p_{\ell \cdot \delta}^-)) > (\ell-1) \delta} = o(1)
	\end{align}
	and
	\begin{align}\label{min:high}
	\Pr{  \delta(G(H,n,p_{\ell \cdot \delta}^+)) < \ell \delta}=o(1).
	\end{align}  
	Proof of \eqref{min:low}: Let $p= p_{\ell \cdot \delta}^-$.
	For $v\in V$ let $I_v= \mathbb{I}\{d(v)=(\ell-1)\delta\}$ and $Z=\sum_{v\in 
	V}I_v$.
	\begin{align*}
	\E{Z} &\geq (1-o(1)) n \binom{n-1}{v_H-1}^{\ell-1} p^{\ell-1} (1-p)^{m_1(H)-\ell+1}
	\\& \geq C_1 n (p n^{(v_H-1)})^{\ell-1} e^{-(p+4p^2)(m_1(H)-\ell+1)}
	\\& \geq C_2 n (\log n)^{\ell-1} e^{-\log n -(\ell-1) \log \log n +\omega(1)} \geq e^{\omega(1)/2}.
	\end{align*}
	In addition, 
	\begin{align*}
	\E{Z^2}&=\sum_{u,v\in V}\Pr{I_{v} \wedge I_{u}} 
	\\& \leq \E{Z}^2+\sum_{u\neq v\in V}\Pr{I_{u} \wedge I_{v} \wedge \{u,v \text{ lie on the same copy of $H$}\}}
	\\ &\leq \E{Z}^2+ \binom{n}{2}\binom{n-2}{r-2} 
	\frac{r!}{\aut(H)}p(1-p)^{(1-o(1))2m_1}
	\\&= \E{Z}^2+ nm_1p(1-p)^{m_1-1}C_3(1-p_2^-)^{(1-o(1))m_1}
	=\E{Z}^2+o(1)\E{Z}
	\\&= (1+o(1))\E{Z}^2.
	\end{align*}
	Chebyshev's inequality give us,
	\[\Pr{|Z-\E{Z}|\geq \E{Z}/2} \leq \frac{\E{Z^2}-\E{Z}^2}{0.25\E{Z}^2}=o(1).\]
	Hence with high probability there exist vertices of degree $(\ell-1)\delta.$
	\vspace{3mm}
	\\Proof of \eqref{min:high}: Let $p= p_{\ell \cdot \delta}^+$. Let $\cE_1$ be 
	the event that in $G(H,n,p)$ there exists a vertex of degree  $d\leq \ell 
	\delta$ that lies on more than $\ell$ copies of $H$. In the event $\cE_1$ 
	there exists a vertex $v$ and a vertex set $S$ of size $d$ such that all the 
	neighbors of $v$ lie in $S$ and at least $\ell+1$ copies of $H$ intersect 
	$S\cup\{v\}$, each in at least $\delta+1$ vertices. Therefore,   
	\begin{align*}
	\Pr{\cE_1}&\leq n\binom{n}{d} [1-p]^{\binom{n-d-1}{v_H-1}}
	\bigg( \binom{d+1}{\delta+1} \binom{n-\delta-1}{v_H-\delta-1} \bigg)^{\ell+1} 
	p^{\ell+1}
	\\ & \leq e^{-p \cdot \binom{n-d-1}{v_H-1}} n^{d+1-\delta (\ell+1)} 
	[n^{v_H-1}p]^{\ell+1}
	\\& \leq  e^{-(1+o(1)) p \cdot m_1(H)} (\log^2 n)^{\delta_d(H)+1}=o(1).
	\end{align*}
	In the event $\neg \cE_1$ the number of vertices of degree less than $\ell 
	\delta$ is bounded by the number of vertices that are covered by at most 
	$\ell-1$ copies of $H$. Thus
	\begin{align*}
	\Pr{\delta(G(H,n,p_{\ell \cdot \delta}^+)) < \ell \delta }& \leq \Pr{\cE_1}+ 
	n \sum_{i=0}^{\ell -1}\binom{m_1(H)}{i}p^i(1-p)^{m_1(H)-i} 
	\\&\leq \ell n (m_1(H)p)^{\ell-1}e^{-pm_1(H)+ p\ell}+o(1)
	\\& \leq \ell n [2\log n]^{\ell-1}e^{-\log n -(\ell-1)\log \log n 
		-\omega(1)}+o(1)=o(1).\qedhere
	\end{align*} 
\end{proof}

\section{Proofs of lemmas for Hamiltoncity}
\begin{proof}[Proof of \Cref{lem:expansion}]
	If there exists  $S\subset \Lrg$ of size $n^{19/20}\leq |S|\leq n/30r$ such that $|N(S)| < 10|S|$ then
	there exist sets $A,B$ of size $n^{19/20}\leq s \leq n/30r$   and  $n-11s$ respectively 
	such that no copy of $H$, $H'$  satisfies $|A \cap H'|=1$ and $|B \cap H'|=r-1$ (take $S=A$ and $B$ to be any subset of $V\setminus(S\cup N(S))$  of size $n-11s$). The probability of such event occurring is bounded above by
	\begin{align*}
	&\sum_{s= n^{19/20}}^{n/30r}\binom{n}{s} 
	\binom{n-s}{10s}(1-p_0)^{\frac{r!}{\aut(H)}\cdot s\binom{n-11s}{r-1}}
	\\ &\leq \sum_{s= n^{19/20}}^{n/30r} \bigg[\frac{en}{s} \cdot 
	\bigg(\frac{en}{10s}\bigg)^{10} 
	e^{-p_0\frac{r!}{\aut(H)}\cdot\binom{n-11s}{r-1}}\bigg]^s
	\\ &\leq \sum_{s= n^{19/20}}^{n/30r} \bigg[  \bigg(\frac{n}{s}\bigg)^{11} e^{- \frac{\ln n-2\ln\ln n}{\binom{n-1}{r-1}}\cdot\binom{n-11s}{r-1}}\bigg]^s
	\\ &\leq \sum_{s= n^{19/20}}^{n/30r} 
	\bigg[  \bigg(\frac{n}{s}\bigg)^{11} 
	\bigg( \frac{\ln^2 n}{n}\bigg)^{ (1-\frac{11s}{n})\cdots (1-\frac{11s-r+2}{n-r+2}) }\bigg]^s
	\\ &\leq \sum_{s= n^{19/20}}^{n/30r} 
	\bigg[  \bigg(\frac{n}{s}\bigg)^{11} 
	\bigg( \frac{\ln^2 n}{n}\bigg)^{ 1-\frac{12sr}{n}}\bigg]^s \leq 
	\sum_{s= n^{19/20}}^{n/30r} 
	\bigg[  n^{11/20} \bigg( \frac{\ln^2 n}{n}\bigg)^{18/30}\bigg]^s=o(1). 
	\end{align*}
	Now assume that there exists a set  
	$S\subset \Lrg$ of size at most $n^{19/20}$ that satisfies
	$|N(S)| < 10|S|$. Since every vertex in $S$ is in at least $\ln \ln n$ copies  of $H$ and every copy of $H$ covers $r$ vertices we have that $S$ intersects at least $|S|\ln \ln n/11$ copies of $H$. Each of those copies is spanned by $S\cup N(S)$. Therefore there exists a set 
	$W \supseteq S\cup N(S)$ of size $w=|W|=11|S|\leq 11n^{19/20}$ that intersects at least $\frac{|W|\ln  \ln n}{11r}$ copies of $H$ each, in at least 2 vertices.
	Since every vertex in $\Lrg$ has $\ln\ln n$ neighbors $|W|\geq \ln \ln n$.
	The probability that such a set exists is bounded by
	\begin{align*}
	&\sum_{w= \ln \ln n}^{11n^{19/20}}
	\binom{n}{w} \binom{r! \binom{w}{2} \binom{n}{r-2}}{w\ln \ln n/11r}  p_0^{w\ln \ln n/11r}\\
	&\leq \sum_{w= \ln \ln n}^{11n^{19/20}} n^w \bigg(\frac{ 11er^3 w n^{r-2} }{\ln \ln n} \bigg)^{w\ln \ln n/11r} p_0^{w\ln \ln n/11r} 
	\\ &\leq \sum_{w= \ln \ln n}^{11n^{19/20}} \bigg[ n^{11r/\ln \ln n} \cdot
	\frac{11er^3 wn^{r-2} }{\ln \ln n} \cdot p_0\bigg]^{w\ln \ln n/11r} 
	\\ &\leq \sum_{w= \ln \ln n}^{11n^{19/20}} \bigg( n^{11r/\ln \ln n} \cdot \frac{w\log n}{n} \bigg)^{w\ln \ln n/11r}
	=o(1).\qedhere
	\end{align*}
\end{proof}

\begin{proof}[Proof of \Cref{smallapart}]
	For $u\in V$ and $Q\subset V$ let $S(u,Q)$ be the event that in $G(H,n,p_0)$ $u$ intersects at most $\ln \ln n$ copies of $H$ that do not intersect  $Q$. For $0\leq |Q|\leq 6$,
	\[\Pr{S(u,Q)}\leq 
	\Pr{\Bin\bigg(\frac{r!}{\aut(H)}\binom{n-7}{r-1},p_0\bigg)\leq \ln \ln n }\leq 
	n^{-0.9}.\]
	Let $\cB$ be the event that for some $u,v \in \Sml$ there exist $\ell\leq 6$ copies of $H$ in $G(H,n,p_2^+)$ that span a connected subgraph containing both $u,v$. If $\cB$ occurs then
	we can find a set 
	$Q=\{v=v_0,v_1,...,v_{\ell-1},v_\ell=u\}$ such that i) the events $S(v,Q\setminus\{v\})$, $S(u,Q\setminus\{u\})$ occur and ii) there exist $H_1,...,H_\ell$ in $G(H,n,p_2^+)$ such that $H_i\cap Q=\{v_{i-1},v_i\}$. Since all the aforementioned events are independent 
	\begin{align*}
	\Pr{\cB}&\leq \sum_{\ell=1}^6 
	\sum_{Q=\{v_0,v_1,...,v_\ell\}}\Pr{S(v_0,Q\setminus\{v_0\} } \cdot \bigg( 
	\binom{n-2}{r-2}\frac{r!}{\aut(H)}p_2^+ \bigg)^{\ell} \cdot 
	\Pr{S(v_\ell,Q\setminus\{v_\ell\} } 
	\\& \leq \sum_{\ell=1}^6 n^{\ell+1} \cdot n^{-0.9} \cdot \bigg(\frac{C_3 \ln 
		n}{n}\bigg)^{\ell} \cdot n^{-0.9} =o(1).\qedhere
	\end{align*}
\end{proof}

\begin{proof}[Proof of \Cref{smallapart2}]
	\Cref{smallapart} implies that w.h.p.\   there do not exist $v\in V$ and 
	$u,w \in \Sml$, $u\neq w$ such that in $G(H,n,p_2^+)$ $v$ and $u$ are in a 
	copy of $H$ and $v$ and $w$ are in a copy of $H$. The probability that there 
	exist $v \in V$, $u \in \Sml\setminus\{v\}$ that are both contained in more 
	than one copy of $H$ in $G(H,n,p_2^+)$ is bounded by
	\begin{equation*}
	\sum_{v,u\in V} \Pr{S(u,\{v\})} \bigg(\binom{n-2}{r-2}\frac{r!}{\aut(H)} 
	p_2^+ \bigg)^2   \leq C_4  n^{-0.9}\log^2 n =o(1).\qedhere
	\end{equation*}
\end{proof}

\begin{proof}[Proof of \Cref{summary}]\ 
  \begin{enumerate}
  \item Recall that we can couple $G(H,n,p_0), \bar{G}(H,n,\tau_2)$ such that 
  $G(H,n,p_0)\subset \bar{G}(H,n,\tau_2)$ w.h.p.\ 
	and there are at least $n \ln \ln n/2r $ copies of $H$ in $\bar{G}(H,n,\tau_2)$ that are not present in $G(H,n,p_0)$.
	From  \cref{smallapart2} it follows that w.h.p.\@ each of those copies that spans a  vertex in $\Sml$ also spans a unique vertex in $V\setminus \Sml$. Hence $w\geq n\ln \ln n/2r-n$.
  \item Let $S\subset V$,  $|S|\leq n/30r$ and set $S_{s}=S\cap \Sml$, 
  $S_{L}=S\cap {\Lrg}$. \Cref{lem:expansion} implies that $|N(S_{L})|\geq 10|S_{L}|$.
	In the case $|S_L|\geq  |S_s|$ we have 
	\[|N(S)| \geq |N(S_L)\setminus S_s| \geq 10|S_L|-|N(S_L)\cap S_s| \geq 10|S_L|-|S_s| \geq 9|S_L| \geq 2|S|.\]
	Next assume $|S_L| < |S_s|$.  
	\Cref{smallapart} implies  that no two vertices in $\Sml$ are within 
	distance 2 in $G(H,n,p_2^+)$, hence their neighborhoods are disjoint. Also $F_0$ has minimum degree 2. Therefore $|N(S_s)|\geq 2|S_s|.$
	Now let $S_L=S_1\cup S_2$ where $S_2$ consists of all the vertices in $S_L$ that are within distance $2$ from $S_s$ and $S_1=S_L\setminus S_2$. 
	If $|S_1|\geq |S_2|$ then since $S_s$ and $S_1$ have disjoint neighborhoods we have that 
	\[|N(S)| \geq |N(S_s)\setminus S_2|+|N(S_1)\setminus S_2| \geq 2|S_s| + 10|S_1|- 2|S_2|  \geq 2|S|.\]
	Otherwise $|S_s| > |S_L|$ and $|S_2| > |S_1|$.
	For $v\in S_s$ let $N_{S_2}(v)$ be the set of vertices in $S_2$ that are within 
	distance $2$ from $v$, hence $\cup_{v\in S_s}N_{S_2}(v)=|S_2|$.  \Cref{smallapart} states that no two vertices in $\Sml$ are within distance 6, thus for $v,u \in S_s, v\neq u$ the sets 
	$N(N_{S_2}(v)), N(N_{S_2}(u))$ are disjoint. In addition since $N_{S_2}\subset S_L$ and $|S_L| \leq |S| \leq n/30r$, \cref{lem:expansion}
	 implies that $|N(N_{S_2}(v))| \geq 10 |N_{S_2}(v)|$ for all $v\in S_s$. Thus 
	\begin{align*}
	|N(S)| &\geq \sum_{v\in S_s} |N(N_{S_2}(v)\cup \{v\})|\\
	&\geq \sum_{v\in S_s} [10|N_{S_2}(v)| - |\{v\}|]\cdot 
	\mathbb{I}_{N_{S_2}(v)\neq \emptyset}
	+|N(v)|\mathbb{I}_{N_{S_2}(v) = \emptyset} 
	\\&\geq \sum_{v\in S_s} 2=2|S_s| \geq |S|.  
	\end{align*}
  \item Assume that there exists a set $S \subset V$ such that $S$ is a 
  connected component of $F_0$ and let $s=|S|$. 
	$F_0$ has minimum degree 2 therefore $s\geq 3$. Let $S_L=S\cap \Lrg$ and $S_s= S\cap \Sml$. \Cref{smallapart2} implies that every vertex in $S_L$ can be adjacent to at most 1 vertex in $\Sml$ hence $|S_L|\geq |S_s|$. Thereafter \cref{lem:expansion} implies that $|S|> n/30r$ since otherwise 
	\[|N(S)|\geq |N(S_L)|-|S_s|\geq 10|S_L|-|S_L|>0.\]
	Finally the probability that there exists a connected component  of size $n/30r\leq s \leq n/2$  in $G(H,n,p_0)\subset F_0$ is bounded by
	\begin{align*}
	\sum_{s=n/30r}^{0.5n}\binom{n}{s}(1-p_0)^{\frac{r!}{\aut(H)}\cdot 
	s\binom{n-s}{r-1}}
	\leq \sum_{s=n/30r}^{0.5n} \bigg[\frac{en}{s}\cdot e^{-C_5 \ln n }\bigg]^s=o(1).
	\end{align*}
	\\iv) First we show that w.h.p.\@ $|\Sml|\leq n^{0.1}$. Indeed by Markov's inequality, 
	\[\Pr{|\Sml|>n^{0.1}}\leq n^{-0.1} \cdot n 
	\Pr{\Bin\bigg(\frac{r!}{\aut(H)}\binom{n-1}{r-1},p_0\bigg)\leq \ln \ln n } 
	=o(1).\]
	Now let  $Q_j$ be a set of 
	$\epsilon n^2$ edges not present in $F_j$ and $Q_j'$ be the subset of $Q_j$ consisting of the edges that are not incident to $\Sml$. Then  
	w.h.p.\@ $|Q_j'| =(1+o(1)) \epsilon n^2$.
	Every  edge in $Q_j'$  belongs to $C_6n^{r-2}$ copies of $H$ that are no present in $F_j$ and every copy of $H$ may cover at most $\binom{r}{2}$ edges in $Q_j'$.
	Therefore there exists a set $W_i$ consisting of at least $C_6n^{r-2} \cdot (1+o(1)) \epsilon n^2 / \binom{r}{2}$ distinct copies of $H$ that intersect $Q_j'$. $H_{i_{j+1}}$ is uniformly distributed among the copies of $H$ that are not present in $ F_j $ and are not incident to a vertex in $\Sml$. Thus
	\begin{equation*}
	\Pr{iv}=\Pr{H_i\in W_i}\geq \frac{C_6n^{r-2} \cdot (1+o(1))\epsilon n^2 / 
		\binom{r}{2}}{ n^r}\geq C_7\epsilon =C_\epsilon .    \qedhere
	\end{equation*}
\end{enumerate}
\end{proof}

\section{Proof sketch of  \Cref{thm:match:th,thm:match:hit}}

To prove Theorem \ref{thm:match:hit} we first indicate the edge set $Q_1$, consisting of   the edges that are incident to vertices of degree 1. Then we delete these edges and the vertex set $U_1$ consisting of the vertices incident to them. Thereafter we use exactly the same techniques as above in order to find a Hamilton cycle in the remaining graph. We use half of the edges of that cycle and the edges in $Q_1$ to form a perfect matching.

Given the above, the only substantial difference  is that while generating $ \bar{G}(H,n,\tau_1)$ (in place of $ \bar{G}(H,n,\tau_2)$) we stop at time $t^*=\min\{i:\delta(G_i')=1\}$. The proofs of all Lemmas with exception the proof of Lemma \ref{summary}, follow in exactly the same way. For the proof of Lemma \ref{summary} 
we have to be slightly more cautious as we want to prove the corresponding statements for the subgraph that is spanned by $V\setminus U_1$. Thus we have to use $\Sml\setminus U_1$ and $\Lrg \setminus U_1$ in place of $\Sml$ and $\Lrg$ respectively.

\section{Proof of \Cref{gamma for path}}

Before proving \Cref{gamma for path}, we derive an expression for $a'/b'$ and establish the following upper bound on $\gamma(a,b,i)$.

\begin{lemma} Consider an $(a,b,i)$ covering of $S$ by a path of length $v-1$ and an $(a',b')$ subcovering with $c'$ connected components. Let $S_j$ be the subgraph of $S$ covered by $j^{th}$ connected component of the $(a',b')$ subcovering. Let $f_j=|E(S_j)|-|V(S_j)|+1$ and $f'=\sum_{j=1}^{c'} f_j$. Let $k$ be the number of duplicate edges in the $(a',b')$ subcovering, i.e. $k$ is the smallest integer such that  removing $k$ edges from multigraph union of $b'$ copies of $H$ can yield a simple graph. Then
	\begin{equation}\label{a'b'}
	\frac{a'}{b'}=v-1+ \frac{c'-f'-k}{b'}
	\end{equation}
	and
	\begin{equation}\label{gamma bd}
	\gamma(a,b,i) \leq 
	\begin{cases}
	v-1 + \frac{1-f}{b} & (a,b,i) \text{ is edge-disjoint}\\
	v-1 - \frac{f}{b} & (a,b,i) \text{ is not edge-disjoint}
	\end{cases}
	\end{equation}
\end{lemma}

\begin{proof}
	We compute $a'$. Note that each of the $b'$ copies of $H$ contributes $v$ vertices, however vertices may be counted multiple times.  We compute
	\[a'= b'v- \brac{b'-\sum_{j=1}^{c'} 1- f_j} -k= b'(v-1) +c'-f'-k,\]
	where the first term subtracted corresponds to doubling counting vertices in each connected component and subtracting $k$ corresponds to removing double counting for vertices adjacent to edges of $S$ that are covered multiple times. 
	
	By definition, $\gamma(a,b,i)\leq a/b$. For the $(a',b')=(a,b)$ subcover that is the entire $(a,b,i)$ cover, $c'=1$, $f'=f$ and $k=0$ if $(a,b,i)$ is edge-disjoint and $k \geq 1$ if $(a,b,i)$ is not edge-disjoint. Thus, \Cref{gamma bd} follows directly from \Cref{a'b'}.
\end{proof}

\begin{proof}
	(of \Cref{gamma for path}.) We consider each case separately. 
	
	\textbf{Case: $f=0$.}
	Consider an $(a,b,i)$ covering. If $(a,b,i)$ is edge-disjoint, then $b \geq \gamma$. It follows from \Cref{gamma bd} that 
	\[\gamma(a,b,i) \leq 
	\begin{cases}
	v-1 + \frac{1}{\beta} & (a,b,i)\text{ is edge-disjoint}\\
	v-1 & (a,b,i)\text{ is not edge-disjoint}.
	\end{cases}
	\]
	Thus $\bar{\gamma}= \max_{(a,b,i) \in \mathcal{I}(S,H)} \gamma(a,b,i) \leq v-1 + 1/\beta$. 
	
	Next consider an edge-disjoint cover of $S$ by $\beta$ copies of $H$, $(a, \beta, i)$. By \Cref{a'b'}, for any $(a',b')$ subcover of the $(a, \beta, i)$ cover, 
	\[\frac{a'}{b'}= v-1 + \frac{c'}{b'}.\] This value is minimized with $c'=1$ and $b'=\beta$, which is achieved by the $(a, \beta)$ subcover which is the whole cover. Thus $\gamma(a, \beta, i)= v-1 + 1/\beta$, and so $\bar{\gamma}\geq  v-1 + 1/\beta$.

	\textbf{Case: $f=1$}. By \Cref{gamma bd}, $\gamma(a,b,i) \leq v-1$ for all $(a,b,i)$ and so it follows that $\bar{\gamma} \leq v-1$. 
	
	Next consider an edge-disjoint cover of $S$, $(a, b, i)$. By \Cref{a'b'}, for any $(a',b')$ subcover of the $(a, \beta, i)$ cover, 
	\[\frac{a'}{b'}= v-1 + \frac{c'-1}{b'}.\] This value is minimized with $c'=1$, which is achieved by the $(a, b)$ subcover which is the whole cover. Thus $\gamma(a, b, i)= v-1 $, and so $\bar{\gamma}\geq  v-1 $.
	
	\textbf{Case: $f\geq 2$}. Consider an $(a,b,i)$ cover. By \Cref{a'b'}, \[\gamma(a,b,i)= \min_{(a',b') \in \mathcal{D}(a,b,i)} \frac{a'}{b'}= \min_{a',b',c',k} v-1 + \frac{c'-f'-k'}{b'}.\]
	Let $t'$ and $e'$ be the number of edges and vertices of $S$ covered by the subcover, so $e'=t'-c'+f'+k$. It follows 
	\begin{equation}\label{three}
	\gamma(a,b,i)= \min_{t',e',b'} v-1 +\frac{t'-e'}{b'}.
	\end{equation}
	
	To give an upper bound on $\gamma(a,b,i)$, we construct a subcover of the $(a,b,i)$ cover as follows. Let $X$ be a subgraph of $S$ with $t^*$ vertices and $e^*$ edges such that $t^*/e^*= \eta$. Let $t',e',b'$ correspond to the subcover that minimally covers $X$, and let $C$ be the subgraph of $S$ covered by this subcover (so $X$ is a subgraph of $C$). 
	
	We claim that $t'-e' \leq t^*-e^*$. Note that $t'-t^*= |V(C) \setminus V(X)|$ and $e'-e^*=|E(C)\setminus E(X)|$. In each component of $C \setminus E(X)$, at least one vertex is included in $V(X)$. Since the number of vertices in a connected component is at least the number of edges in the connected component minus one, and at least one vertex in each connected component is not included in $V(C) \setminus V(X)$, it follows that $|V(C) \setminus V(X)| \geq |E(C) \setminus E(X)|$. Thus $t'-t^* \leq e'- e^*$ and the claim follows.
	
	By considering this subcover with parameters $t',e',b'$, we obtain 
	\[\gamma(a,b,i) \leq v-1 + \frac{t'-e'}{b'} \leq v-1 +\frac{t^* - e^*}{e^*} = v-2 + \eta\] since $b' \leq e^*$ and $t^*-e^* \leq 0$. It follows that $\bar{\gamma} \leq v-2 + \eta$.
	
	Finally to lower bound $\bar{\gamma}$ we consider a cover in which there are $b=|E(S)|$ copies of $H$ and each copy covers precisely one edge of $S$. In this case in all subcovers $b'=e'$. By \Cref{three}
	\[\gamma(a,b,i) = \min_{t',e',b'} v-1 + \frac{t'-e'}{b'}= \min_{t',e'} v-2 + \frac{t'}{e'}= v-2 - \eta.\] Thus $\bar{\gamma} \geq v-2 + \eta$.
\end{proof}

\end{document}